\documentclass[a4paper,oneside]{article}
\usepackage[utf8]{inputenc}
\usepackage[a4paper,margin=3cm]{geometry} 
\usepackage{amsmath}
\usepackage{amsthm}
\usepackage{amscd}
\usepackage{amssymb}
\usepackage{MnSymbol}
\usepackage{latexsym}
\usepackage{eucal}
\usepackage{dsfont}
\usepackage{mathtools}
\usepackage{enumitem}

\usepackage[colorlinks,pdftex]{hyperref}
\hypersetup{
linkcolor=black,
citecolor=black,
pdftitle={Moderate deviations and Strassen's law for additive processes},
pdfauthor={Franziska K\"uhn, Ren\'e L. Schilling},
pdfkeywords={moderate deviations, Strassen's law, additive processes, functional law of the iterated logarithm},
}

\makeatletter
\renewcommand\section{\@startsection{section}{1}{0mm}{-1.5\baselineskip}{\baselineskip}{\normalsize\bfseries\sffamily}}
\renewcommand\subsection{\@startsection{subsection}{1}{0mm}{-\baselineskip}{\baselineskip}{\normalsize\bfseries\sffamily}}
\makeatother

\makeatletter
\def\@fnsymbol#1{\ensuremath{\ifcase#1\or *\or **\or \dagger\or \ddagger\or
   \mathsection\or \mathparagraph\or \|\or \dagger\dagger
   \or \ddagger\ddagger \else\@ctrerr\fi}}

\newlength{\preskip}
\newlength{\postskip}
\makeatother

\newtheoremstyle{theorem}{\preskip}{\postskip}{\itshape}{}{\bfseries}{}
{.5em}{\textbf{\thmname{#1}\thmnumber{ #2} (\thmnote{ #3})}}
\newtheoremstyle{definition}{\preskip}{\postskip}{\normalfont}{0pt}{\bfseries}{}{.5em}{}
\newtheoremstyle{remark}{\preskip}{\postskip}{\normalfont}{0pt}{\bfseries}{}{.5em}{}

\swapnumbers
\theoremstyle{theorem} \newtheorem{thm}{Theorem}[section]
\theoremstyle{theorem} \newtheorem{lem}[thm]{Lemma}
\theoremstyle{theorem} 
\theoremstyle{theorem} \newtheorem{kor}[thm]{Corollary}
\theoremstyle{definition} 
\theoremstyle{definition} 
\theoremstyle{definition} \newtheorem*{ack}{Acknowledgements}
\theoremstyle{remark} 
\theoremstyle{remark} 
\theoremstyle{definition}  \newtheorem{bsp}[thm]{Example}
\theoremstyle{definition}  
\DeclareMathOperator \var {Var}

\DeclareMathOperator \sgn {sgn}

\DeclareMathOperator \bv {BV[0,1]}

\DeclareMathOperator \dom {Dom}

\newcommand\komp[1]{\langle#1\rangle}

\newcommand{\I}{\mathds{1}}
\newcommand\floor[1]{\left\lfloor #1 \right\rfloor}
\newcommand\fa{\qquad \text{for all \ }}

\newcommand{\cadlag}{c\`adl\`ag }

\newcommand{\als}{\qquad \text{a.\,s.}}

\newcommand\mc[1] {\mathcal{#1}}
\newcommand\mbb[1] {\mathds{#1}}

\newcommand{\eps}{\varepsilon}
\newcommand{\ac}{AC[0,1]}

\newcommand*{\lcdot}{\,\raisebox{-0.25ex}{\scalebox{1.4}{$\cdot$}}\,}

\author{%
    Franziska K\"{u}hn\thanks{Institut f\"ur Mathematische Stochastik, Fachrichtung Mathematik, Technische Universit\"at Dresden, 01062 Dresden, Germany, \texttt{franziska.kuehn1@tu-dresden.de}} 
\and 
    Ren\'e L.\ Schilling\thanks{Institut f\"ur Mathematische Stochastik, Fachrichtung Mathematik, Technische Universit\"at Dresden, 01062 Dresden, Germany, \texttt{rene.schilling@tu-dresden.de}}
}

\title{Moderate deviations and Strassen's law for additive processes}

\date{}

\begin{document}

\maketitle

\abstract{\noindent 
    We establish a moderate deviation principle for processes with independent increments under certain growth conditions for the characteristics of the process. Using this moderate deviation principle, we give a new proof for Strassen's functional law of the iterated logarithm. In particular, we show that any square-integrable L\'evy process satisfies Strassen's law. 
\par\medskip

 \par \medskip

}

\section{Introduction} \label{s-intro}
Let $(X_t)_{t \geq 0}$ be a L\'evy process such that $\mathbb{E}X_t=0$, and assume that the weak Cram\'er condition holds, i.\,e.\ $\mbb{E}e^{\lambda |X_1|}<\infty$ for some $\lambda>0$. It is known (see e.\,g.\ Mogulskii \cite{mog93} or Feng--Kurtz \cite{feng}) that the family of scaled L\'evy process $(X(t \lcdot)/S(t))_{t>0}$ obeys a moderate deviation principle in the space of \cadlag functions $(D[0,1],\|\cdot\|_{\infty})$ with good rate function \begin{equation*}
		I(f) := \begin{cases} \displaystyle\frac{1}{2 \var X_1} \int_0^1 f'(s)^2 \, ds, & f :[0,1] \to \mbb{R}\,\, \text{absolutely continuous}, f(0)=0, \\ \displaystyle\infty, & \text{otherwise}, \end{cases}
	\end{equation*}
	and speed $S(t)^2/t$ if the scaling function $S$ satisfies \begin{equation*}
		\frac{S(t)}{\sqrt{t}} \xrightarrow[]{t \to \infty} \infty \quad \text{and} \quad \frac{S(t)}{t} \xrightarrow[]{t \to \infty} 0.
	\end{equation*}
	The assumptions on the moments of $X_1$ have been substantially weakened by Gao \cite{gao05}. In this paper, we consider the corresponding moderate deviation principle for additive processes, see Section~\ref{s-def}. Large deviation results for additive processes have been obtained by Puhalskii \cite{puc94} and Liptser--Puhalskii \cite{lip-puc} for the scaling $S(t)=t$ under rather abstract conditions on the characteristics and the stochastic exponential, respectively.
	We state sufficient conditions in terms of the growth of the characteristics and give a direct proof of the moderate deviation principle using the well-known G\"artner--Ellis theorem. In particular, we obtain two representations for the good rate function $I$. \par
	As an application, we study Strassen's law for additive processes. Strassen \cite{stra64} proved that for a (one-dimensional) Brownian motion $(B_t)_{t \geq 0}$ the set \begin{equation*}
		\left\{ f \in C[0,1];\; \exists t>0: f(s)=\frac{B(t s)}{\sqrt{2t \log \log (t \vee e^e)}} \, \, \text{for all} \, s \in [0,1] \right\}
	\end{equation*}
	is almost surely relatively compact in $(C[0,1],\|\cdot\|_{\infty})$ and its limit points are given by \begin{equation*}
		\left\{f: [0,1] \to \mathbb{R};\; f(0)=0, f \,\, \text{absolutely continuous}, \int_0^1 f'(s)^2 \, ds  \leq 1 \right\};
	\end{equation*}
	this result is called Strassen's law (or functional law of the iterated logarithm). There have been several generalizations since then; Wichura \cite{wic73} and Wang \cite{wang} studied Strassen's law for additive processes and Maller \cite{mal08} obtained a small-time version of Strassen's law for L\'evy processes. More recently, Gao \cite{gao09} proved a Strassen law for a subclass of locally square-integrable martingales. Strassen's law has a variety of applications, e.\,g.\ functional limit theorems such as the law of the iterated logarithm; see e.\,g.\ Strassen \cite{stra64}, Wichura \cite{wic73} and Buchmann et al. \cite{bms08}. \par
	We show that, whenever a certain moderate deviation principle holds for an additive processes, the process satisfies a functional law of the iterated logarithm. Combining this result with the moderate deviation principle proved in the first part of this paper, this covers the corresponding results in Wang \cite{wang}, cf.\ Corollary~\ref{add-n05}, and Wichura \cite{wic73}. As a special case, we find that any square-integrable L\'evy process satisfies Strassen's law.
	Let us emphasize that we do not intend to formulate the results in the most general form but to present an alternative proof for Strassen's law. The results discussed here are generalizations of these obtained by K\"uhn \cite{fk14}. \par
	The paper is organized as follows. In Section~\ref{s-def}, we introduce basic definitions and notation. The main results are stated in Section~\ref{s-main} and proved in Section~\ref{s-proof}. Finally, in Section~\ref{s-rem}, we sketch some generalizations.

\section{Basic definitions and notation} \label{s-def}
	Let $(X_t)_{t \geq 0}$ be a (real-valued) stochastic process on a complete probability space $(\Omega,\mc{A},\mbb{P})$. We call $(X_t)_{t \geq 0}$ an \emph{additive process} if $(X_t)_{t \geq 0}$ has \cadlag sample paths, independent increments and $X_0=0$. If, additionally, $(X_t)_{t \geq 0}$ has stationary increments, i.\,e.\ $X_t-X_s \sim X_{t-s}-X_0$, $s \leq t$, then $(X_t)_{t \geq0}$ is a \emph{L\'evy process}. Any mean-zero square-integrable additive process $(X_t)_{t \geq 0}$ admits a \emph{L\'evy--It\^o decomposition} of the form
\begin{equation*}
		X(t) = X^c(t) + \int_0^t \!\! \! \int z \, (N(dz,dr)-\nu(dz,dr)), \qquad t \geq 0,
	\end{equation*}
	where $X^c$ denotes the continuous martingale part, $N$ the jump measure of $(X_t)_{t \geq 0}$ and $\nu$ its compensator. Moreover, there exist increasing deterministic functions $A,C$ and a family of $\sigma$-finite measures $K_r$ on $(\mbb{R},\mc{B}(\mbb{R}))$ such that $C_t \geq 0$, $\int (z^2 \wedge 1) K_r(dz) \leq 1$, and \begin{equation}
		\komp{X^c}_t = C_t, \qquad  \qquad \nu(dz,dr) = K_r(dz) \, dA_r. \label{add-neq05}
	\end{equation}
	If $(X_t)_{t \geq 0}$ is a L\'evy process, then $A$ and $C$ are linear, and $K_r$ does not depend on $r$. We denote by  \begin{equation*}
		[\nu]_{s,t} := \inf\left\{\varrho>0;\; \nu(B(0,\varrho)^c \times [s,t])=0\right\}, \qquad (\inf \emptyset := \infty)
	\end{equation*}
	the maximal jump height during the time interval $[s,t]$. The process $(X_t)_{t \geq 0}$ has no \emph{fixed jump discontinuities} if $\nu(\mbb{R} \times \{t\} ) =0$ for any $t>0$. Our standard reference for additive processes is the monograph by Jacod--Shiryaev \cite{jac}, we use Sato \cite{sato} for L\'evy processes. \par
	For simplicity, we assume that the mapping $(\Omega,\mc{A}) \ni \omega \mapsto X(\lcdot,\omega) \in (D[0,1],\|\cdot\|_{\infty})$ is measurable. Here $D[0,1]$ denotes the space of \cadlag (i.\,e.\ right-continuous with left-hand limits) functions $f:[0,1] \to \mbb{R}$ endowed with the uniform norm $\|\cdot\|_{\infty}$.  By $\ac$ we denote the set of absolutely continuous functions $f:[0,1] \to \mbb{R}$. \par
	Recall that a family $(X^t)_{t>0}$ of stochastic processes with values in a metric space $(M,d)$ satisfies a  \emph{large deviation principle in $(M,d)$ with good rate function $I:M \to [0,\infty]$ and speed $(a_t)_{t>0} \subseteq (0,\infty)$} if $I$ has compact sublevel sets $\Phi(r):=\{f \in M; I(f) \leq r\}$, $a_t \to \infty$ as $t \to \infty$, and \begin{equation*}
		- \inf_{f \in U} I(f) \leq \liminf_{t \to \infty} \frac{1}{a_t} \log \mbb{P}(X^t \in U), \qquad
		\limsup_{t \to \infty} \frac{1}{a_t} \log \mbb{P}(X^t \in F) \leq - \inf_{f \in F} I(f)
	\end{equation*}
	holds for any open set $U \subseteq M$ and closed set $F \subseteq M$, respectively. Let $(X_t)_{t \geq 0}$ be a square-integrable stochastic process with \cadlag sample paths and $S:(0,\infty) \to (0,\infty)$ such that $S(t)/\var X_t \to 0$ as $t \to \infty$. We say that $(X(t \lcdot)/S(t))_{t>0}$ satisfies a \emph{moderate deviation principle with good rate function $I$ and speed $(a_t)_{t>0}$} if the family satisfies a large deviation principle in $(D[0,1],\|\cdot\|_{\infty})$ with good rate function $I$ and speed $(a_t)_{t>0}$.
	For a detailed discussion of large deviation theory, we refer the reader to the monographs by Dembo--Zeitouni \cite{dz} and Feng--Kurtz \cite{feng}.
	
\section{Main results} \label{s-main}

We are now in the position to state the main results.

\begin{thm} \label{add-n01}
	Let $(X_t)_{t \geq 0}$ be an additive process without fixed jump discontinuities such that $\mathbb{E}X_t = 0$ and \begin{equation}
		\lim_{t \to \infty} \frac{\var X_t}{t^\gamma} =: \sigma^2>0\label{add-neq11}
	\end{equation}
	exists for some $\gamma >0$. Let $S:(0,\infty) \to (0,\infty)$ be such that \begin{equation}
			\frac{S(t)}{t^{\gamma/2}} \xrightarrow[]{t \to \infty} \infty, \qquad \qquad \frac{S(t)}{t^{\gamma}} \xrightarrow[]{t \to \infty} 0, \label{add-neq13}
		\end{equation}	
	and \begin{equation}
		 \frac{S(t)}{S(\floor{t})} \xrightarrow[]{t \to \infty} 1. \label{add-neq15}
	\end{equation}
	Suppose that one of the following conditions holds.
	\begin{enumerate}[label*=\upshape(C\arabic*),ref=\upshape(C\arabic*)]
		\item\label{C2}  There exists a measure $G$ on $(\mbb{R},\mc{B}(\mbb{R}))$ such that $\int_{|z|>1} e^{2\lambda_0 |z|} \, G(dz)<\infty$ for some $\lambda_0>0$ and \begin{equation*}
			K_t(B(0,r)^c) \leq G(B(0,r)^c) \fa  t \geq 0,\: r \geq 1.
		\end{equation*}
		Furthermore,
		\begin{equation}
				\lim_{t \to \infty} |A_t| \frac{S(t)}{t^{2\gamma}}= 0. \label{add-neq19}
		\end{equation}
		\item\label{C1} \strut \vspace{-\baselineskip} \begin{equation}
				\lim_{t\to \infty} [\nu]_{0,t} \frac{S(t)}{t^\gamma} =0 \label{add-neq17}
		\end{equation}
		(In this case, we implicitly set $\lambda_0 := \infty$.)
	\end{enumerate}
	Then $(X(t\lcdot)/S(t))_{t > 0}$ satisfies a moderate deviation principle in $(D[0,1],\|\cdot\|_{\infty})$ with speed $S(t)^2/t^\gamma$ and good rate function \begin{equation}
		I(f) := \begin{cases} \displaystyle\frac{1}{2\sigma^2 \gamma} \int_0^1 \frac{f'(s)^2}{s^{\gamma-1}} \, ds, & f \in \ac,\: f(0)=0, \\ \displaystyle\infty, & \text{otherwise}. \end{cases} \label{add-neq21}
	\end{equation}
\end{thm}

In particular, if $(X_t)_{t \geq 0}$ is a mean-zero L\'evy process such that $\mbb{E}e^{\lambda |X_1|}<\infty$ for some $\lambda>0$, then condition \ref{C2} holds with $G=\nu$ where $\nu$ is the L\'evy measure of $(X_t)_{t \geq 0}$, cf.\ \cite[Theorem 25.17]{sato}.

\begin{thm}[Strassen's law] \label{add-n03}
	Let $(X_t)_{t \geq 0}$ be an additive process and $S(t)= \sqrt{2t^{\gamma} \log \log (t \vee e^e)}$ for some $\gamma>0$. If $(X(t \lcdot)/S(t))_{t \geq 0}$ satisfies a moderate deviation principle in $(D[0,1],\|\cdot\|_{\infty})$  with speed $S(t)^2/t^{\gamma}$ and good rate function \begin{equation}
		I(f) = \begin{cases} \displaystyle\frac{1}{2 \sigma^2 \gamma} \int_0^1 \frac{f'(s)^2}{s^{\gamma-1}} \, ds, & f \in \ac,\: f(0)=0, \\ \displaystyle\infty, & \text{otherwise}, \end{cases} \label{add-neq23}
	\end{equation}
	for some $\sigma>0$, then \begin{equation*}
		\left\{ \frac{X(t \lcdot)}{\sqrt{2 t^{\gamma}\log \log (t \vee e^e)}}; t > 0\right\}
	\end{equation*}
	is a.\,s.\ relatively compact in $(D[0,1],\|\cdot\|_{\infty})$ as $t \to \infty$, and the set of limit points $\mc{L}(\omega)$ (as $t \to \infty$) is for almost all $\omega \in \Omega$ given by the sublevel set $\Phi(\frac{1}{2})=\{f \in D[0,1]; I(f) \leq \frac{1}{2}\}$ of the good rate function $I$.
\end{thm}

By a standard argument we can replace $t^{\gamma}$ by a regulary varying function $h(t)$ of index $\gamma>0$, cf.\ Section~\ref{s-rem} at the end of the paper. \par
Note that the \emph{law of the iterated logarithm} is a simple consequence of Strassen's law: \begin{equation*}
	\limsup_{t \to \infty} \frac{X_t}{\sqrt{2 t^{\gamma} \log \log t}} = \sigma \als
\end{equation*}
Moreover, if $(X_t)_{t \geq 0}$ is an additive process such that $\var X_t/t^{\gamma} \xrightarrow[]{t \to \infty} \sigma^2>0$ and one of the growth conditions \ref{C1} or \ref{C2} holds, Theorem~\ref{add-n01} shows that Theorem~\ref{add-n03} is applicable.

\begin{kor}\label{add-n05}
	Let $(X_t)_{t \geq 0}$ be an additive process such that $\var X_t/t \xrightarrow[]{t \to \infty} \sigma^2>0$. Suppose that there exists a finite measure $G$ on $(\mbb{R},\mc{B}(\mbb{R}))$ such that \begin{equation*}
		K_t(B(0,r)^c) \leq G(B(0,r)^c) \fa r \geq 1, t \geq 0 \quad \text{and} \quad \int_{\mbb{R}} z^2 \, G(dz)< \infty.
	\end{equation*}
	Then $(X_t)_{t \geq 0}$ satisfies Strassen's law (with $\gamma=1$). In particular, Strassen's law holds for any square-integrable L\'evy process $(X_t)_{t \geq 0}$.
\end{kor}

We close this section with two examples which we state for the more general case of regulary varying functions, see Section~\ref{s-rem} and the remark before Corollary~\ref{add-n05}.

\begin{bsp} \label{add-n065}
	Let $(L_t)_{t \geq 0}$ be a pure-jump L\'evy process,  \begin{equation*}
		L_t = \int_0^t \!\!\! \int z \, (N(dz,ds)-\nu_L(dz) \, ds),
	\end{equation*}
	such that $\var L_1=\sigma^2<\infty$. For a non-decreasing function $\alpha:[0,\infty) \to [0,\infty)$, $\alpha(t) \xrightarrow[]{t \to \infty} \infty$, we define a truncated L\'evy process $(X_t)_{t \geq 0}$ by \begin{equation*}
		X_t := \int_0^t \!\!\! \int z \I_{\{|z| \leq \alpha(s)\}} \, (N(dz,ds)-\nu_L(dz) \, ds).
	\end{equation*}
	From \begin{equation*} 
		0 \leq \var L_t-\var X_t = \int_0^t \!\!\! \int z^2 \I_{\{|z| >\alpha(s)\}} \, \nu_L(dz) \, ds
	\end{equation*}
	it follows easily that \begin{equation*}
		\lim_{t \to \infty} \frac{\var X_t}{t} = \lim_{t \to \infty} \frac{\var L_t}{t} = \sigma^2.
	\end{equation*}
	By Corollary~\ref{add-n05}, $(X_t)_{t \geq 0}$ satisfies Strassen's law (with $\gamma=1$). Let us remark that truncated processes of this form are used in \cite{wang} to prove Corollary~\ref{add-n05}.
\end{bsp}

\begin{bsp} \label{add-n06} 
	Let $(L_t)_{t \geq 0}$ be a L\'evy process with bounded jumps such that $\mathbb{E}L_1^2 = \sigma_L^2 < \infty$ and $\mathbb{E}L_1=0$. Denote by $\psi$ its characteristic exponent and $\nu_L$ its L\'evy measure. For a regulary varying function $\alpha:[0,\infty) \to [0,\infty)$ of index $\gamma>0$, the additive process $X_t := \int_0^t \alpha(s) \, dL_s$ satisfies a moderate deviation principle with speed $S(t)^2/(t \alpha(t)^2)$ and good rate function \begin{equation*}
		I(f) = \begin{cases}\displaystyle \frac{2\gamma+1}{2\gamma \sigma_L^2} \int_0^1 \frac{f'(s)^2}{s^2 \alpha(s)^2} \,ds, & f \in \ac, f(0)= 0, \\ \infty, & \text{otherwise}, \end{cases}
	\end{equation*}
	for any scaling function $S:(0,\infty) \to (0,\infty)$ such that \eqref{add-neq15} as well as the growth conditions $S(t)/\sqrt{t\alpha(t)} \xrightarrow[]{t \to \infty} \infty$ and $S(t)/(t \alpha(t)) \xrightarrow[]{t \to \infty}0$ hold. Moreover, the set
	\begin{equation*}
		\left\{ \frac{X(t \lcdot)}{\sqrt{2 t \alpha(t)^2 \log \log (t \vee e^e)}}; t > 0\right\}
	\end{equation*}
	is a.\,s.\ relatively compact in $(D[0,1],\|\cdot\|_{\infty})$ as $t \to \infty$, and the set of limit points (as $t \to \infty$) is for almost all $\omega \in \Omega$ given by the sublevel set $\Phi(\frac{1}{2})$. \emph{Indeed:} Using that $(L_t)_{t \geq 0}$ is a martingale with independent increments, it is not difficult to see that \begin{equation*}
		\var X_t = \sigma_L^2 \int_0^t \alpha(s)^2 \, ds \quad \Longrightarrow \quad \frac{\var X_t}{t \alpha(t)^2} \xrightarrow[]{t \to \infty} \frac{\sigma_L^2}{2\gamma+1}.
	\end{equation*}
	The approximation \begin{equation*}
		X_t = \int_0^t \alpha(s) \, dL_s \approx \sum_{j=1}^n \alpha(s_j) (L_{s_j}-L_{s_{j-1}})
	\end{equation*}
	shows that the characteristic function of $X_t$ equals $\exp\left(-\int_0^t \psi(\alpha(s) \xi) \, ds \right)$. In particular, \begin{equation*}
		\nu(dz,ds) \stackrel{\eqref{add-neq05}}{=} K_s(dz) \, dA_s = \nu_L\left(\frac{1}{\alpha(s)} dz \right) \, ds.
	\end{equation*}
	Note that due to the boundedness of the jumps of $(L_t)_{t \geq 0}$, condition \ref{C1} is satisfied. Therefore, the claim follows from Theorem~\ref{add-n01} and Theorem~\ref{add-n03}. 
\end{bsp}

\section{Proofs} \label{s-proof}
	We start with the proof of the moderate deviation principle, Theorem \ref{add-n01}, and split the proof into several steps: \begin{enumerate}
		\item The sequence of discretizations $(Z_n/S(n))_{n \in \mbb{N}}$ defined by\begin{equation*}
			\frac{Z_n(s,\omega)}{S(n)}
			:= \frac{1}{S(n)} X(\floor{n s},\omega)
			= \frac{1}{S(n)} \left(\sum_{j=0}^{n-1} X(j,\omega) \I_{[j/n,(j+1)/n)}(s) +  X(n,\omega) \I_{\{1\}}(s) \right)
		\end{equation*}
		is exponentially tight in $(D[0,1],\|\cdot\|_{\infty})$, cf.\ Lemma~\ref{add-n09}.
		\item $(Z_n/S(n))_{n \in \mbb{N}}$ satisfies a moderate deviation principle in $(D[0,1],\|\cdot\|_{\infty})$ with good rate function $J$, \begin{equation}
			J(f) := \sup_{\alpha \in \bv \cap D[0,1]} \left( \int_0^1 f \, d\alpha - \frac{\gamma \sigma^2}{2} \int_0^1 s^{\gamma-1} (\alpha(1)-\alpha(s))^2 \, ds \right), \label{add-neq31}
		\end{equation}
		and speed $a_n = S(n)^2/n^{\gamma}$, cf.\ Theorem~\ref{add-n11}; as usual, $\bv$ denotes the set of functions $\alpha:[0,1] \to \mbb{R}$ of bounded variation.
		\item $(Z_{\floor{t}}/S(\floor{t}))_{t > 0}$ and $(X(t\lcdot)/S(t))_{t >0}$ are exponentially equivalent, cf.\ Lemma~\ref{add-n13}.
		\item The good rate function $J$ equals $I$ defined in \eqref{add-neq21}, cf.\ Theorem~\ref{add-n15}.
	\end{enumerate}
	We essentially follow the lines of de Acosta \cite{acosta-lp}. For the readers' convenience, we include the proofs of (i)-(iv). The next lemma provides an estimate for the exponential moments of $X_t-X_s$.

\begin{lem} \label{add-n07}
	Let $(X_t)_{t \geq 0}$ be as in Theorem~\ref{add-n01}. Then \begin{equation}
		\mbb{E}e^{\lambda (X_t-X_s)} \leq \exp \left( \frac{1}{2} \lambda^2 \var(X_t-X_s) + |\lambda|^3 E_{s,t}\right) \label{add-neq33}
	\end{equation}
	for any $|\lambda| \leq \lambda_0$ and $s \leq t$ where  \begin{equation*}
		E_{s,t} := E_{s,t}(\lambda):=\begin{cases} \frac{1}{6} [\nu]_{s,t} \var(X_t-X_s) e^{|\lambda| [\nu]_{s,t}} , & \text{if \ref{C1} holds}, \\ \frac{1}{6}  |A_t-A_s| \cdot \left(e^{\lambda_0}+\int_{|z|>1} e^{\lambda_0 |z|} |z|^3 \, G(dz) \right), & \text{if \ref{C2} holds}. \end{cases}
	\end{equation*}
	In particular, \begin{equation}
		\lim_{t \to \infty} \frac{S(t)}{t^{2\gamma}} E_{0,t} \left( r \frac{S(t)}{t^\gamma} \right) = 0 \fa r \geq 0. \label{add-neq35}
	\end{equation}
\end{lem}

\begin{proof}
	It follows from the conditions \ref{C1} or \ref{C2}, respectively, that $\mathbb{E}e^{\lambda (X_t-X_s)}<\infty$ for any $s \leq t$, $|\lambda| \leq \lambda_0$, and that  \begin{equation*}
		\mbb{E}e^{\lambda (X_t-X_s)}
		= \exp \left( \frac{1}{2} (C_t-C_s) \lambda^2 + \int_s^t \!\!\! \int (e^{\lambda z}-1-\lambda z)\, \nu(dz,dr)\right),
	\end{equation*}
	cf.\ Fujiwara \cite{fuji}. Since \begin{equation*}
		\var(X_t-X_s) = (C_t-C_s)+ \int_s^t \!\!\! \int z^2 \, \nu(dz,dr)
	\end{equation*}
	Taylor's formula yields \begin{equation*}
		\mbb{E}e^{\lambda (X_t-X_s)}
		\leq \exp \left( \frac{\lambda^2}{2} \var(X_t-X_s) + \frac{|\lambda|^3}{6} \int_s^t \!\!\! \int |z|^3 e^{\lambda \xi} \, \nu(dz,dr) \right)
	\end{equation*}
	for some intermediate value $\xi=\xi(z) \in (0,z)$. From the definition of $[\nu]_{s,t}$ we get \begin{align*}
		\int_s^t \!\!\! \int |z|^3 e^{\lambda \xi} \, \nu(dz,dr)
		&\leq [\nu]_{s,t} \cdot e^{|\lambda| [\nu]_{s,t}} \int_s^t \!\!\! \int z^2 \, \nu(dz,dr)
		\leq [\nu]_{s,t} \var(X_t-X_s) e^{|\lambda| [\nu]_{s,t}} .
	\end{align*}
	This proves \eqref{add-neq33} if \ref{C1} holds. If \ref{C2} is satisfied, the claim follows from the estimate \begin{align*}
		\int_s^t \!\!\! \int |z|^3 e^{\lambda \xi} \, \nu(dz,dr)
		&\leq \int_s^t \!\!\! \int |z|^3 e^{|\lambda| |z|} \, K_r(dz) \, dA_r \\
		&\leq \left(e^{\lambda_0}+ \int_{|z| > 1} |z|^3 e^{\lambda_0 |z|} \, G(dz) \right) |A_t-A_s|.
	\end{align*}
	\eqref{add-neq35} is a direct consequence of the definition of $E_{0,t}$ and the assumptions in \ref{C1} and \ref{C2}, respectively.
\end{proof}

In order to show that the approximations $(Z_n)_{n \in \mbb{N}}$ satisfy a moderate deviation principle, we need the following lemma.

\begin{lem} \label{add-n09}
	For each $n \in \mbb{N}$, $Z_n/S(n)$ is tight in $(D[0,1],\|\cdot\|_{\infty})$. Morover, $(Z_n/S(n))_{n \in \mbb{N}}$ is exponentially tight, i.\,e.\ for any $R \geq 0$ there exists a compact set $K \subseteq D[0,1]$ such that \begin{equation*}
		\limsup_{n \to \infty} \frac{n^{\gamma}}{S(n)^2} \log \mbb{P}\left( \frac{Z_n}{S(n)} \notin K \right) \leq -R.
	\end{equation*} \end{lem}

\begin{proof}
	Since the mapping \begin{equation*}
		(\mbb{R}^n,\|\cdot\|) \ni x \mapsto (T_n x)(t) := \sum_{j=1}^{n-1} x_j \I_{[j/n,(j+1)/n)}(t) + x_n \I_{\{1\}}(t) \in (D[0,1],\|\cdot\|_{\infty})
	\end{equation*}
	is continuous, it follows that $T_n(K)$ is compact for any compact set $K \subseteq \mbb{R}^n$. For $K \subseteq \mbb{R}$ compact and $K^n = K \times \ldots \times K \subseteq \mbb{R}^n$, we have \begin{equation*}
		\mbb{P} \left( \frac{Z_n}{S(n)} \notin T_n(K^n) \right) \leq \sum_{j=1}^n \mbb{P} \left( \frac{X_j}{S(n)} \notin K \right).
	\end{equation*}
	Since $X_j/S(n)$ is tight for $j=1,\ldots,n$, we conclude that $Z_n/S(n)$ is tight in $(D[0,1],\|\cdot\|_{\infty})$.  It remains to prove exponential tightness. To this end, we show that the assumptions of \cite[Lemma 3.3]{feng} are satisfied. Fix $r > 0$ and $\eps>0$. For $K \subseteq \mbb{R}$ and $n \geq m$, we have \begin{align}
		&\mbb{P} \left( d \left( \frac{Z_n}{S(n)}, T_m(K^m) \right) > \eps \right) \notag \\
		&\quad \leq \mbb{P}\left( \frac{Z_n}{S(n)} \notin T_n(K^n) \right) + \mbb{P} \left( \frac{Z_n}{S(n)} \in T_n(K^n), d \left( \frac{Z_n}{S(n)}, T_m(K^m) \right)>\eps \right)
		=: I_1+I_2 \label{add-neq37}
	\end{align}
	with $d(f,A) := \inf_{g \in A} \|f-g\|_{\infty}$, $A \subseteq D[0,1]$. We choose $K:=[-r,r]$ and estimate the terms separately. Applying Etemadi's inequality, Markov's inequality and Lemma~\ref{add-n07} yields \begin{align*}
		I_1
		= \mbb{P} \left( \max_{1 \leq j \leq n} \left|\frac{X_j}{S(n)}\right|>r \right)
		&\leq 3 \max_{1 \leq j \leq n} \mbb{P} \left( |X_j| > \frac{S(n) r}{3} \right)  \\
		&\leq 3 \exp \left(- \frac{S(n) \lambda r}{3} \right) \max_{1 \leq j \leq n} \bigg( \mbb{E}e^{\lambda X_j}+\mbb{E}e^{-\lambda X_j} \bigg) \\
		&\leq 6 \exp \left(- \frac{S(n) \lambda r}{3} \right) \exp \left( \frac{\lambda^2}{2} \var X_n + \lambda^3 E_{0,n}(\lambda) \right)
	\end{align*}
	for any $0 \leq \lambda \leq \lambda_0$. For $\lambda := S(n)/n^\gamma$, we have $\lambda \leq \lambda_0$ for $n$ sufficiently large, and we obtain \begin{equation*}
		I_1 \leq 6 \exp \left(- \frac{S(n)^2}{n^\gamma} \cdot \frac{r-3\sigma^2}{3} \right)
	\end{equation*}
	since \begin{equation*}
		\lim_{n \to \infty} \left(\frac{1}{2}\frac{\var X_n}{n^\gamma} +\frac{S(n)}{n^{2\gamma}} E_{0,n} \left( \frac{S(n)}{n^\gamma} \right) \right) = \frac{\sigma^2}{2},
	\end{equation*}
	cf.\ \eqref{add-neq11} and \eqref{add-neq35}. In order to estimate $I_2$ we observe that for $f_m := f(\floor{m \lcdot}/m)$ it holds that\begin{equation}
		d(f,T_m(K^m))
		\leq \|f-f_m\|_{\infty} \label{add-neq41} \fa f \in T_n(K^n).
	\end{equation}
	In abuse of notation, we write $x+y \wedge z := (x+y) \wedge z$. Then, \begin{align}
		\|f-f_m\|_{\infty}
		&= \adjustlimits\max_{0 \leq i \leq m-1} \sup_{t \in [i/m,(i+1)/m)} \left| f \left( \frac{\floor{n t}}{n} \right)- f \left( \frac{\floor{m t}}{m} \right) \right| \notag \\
		&\leq \adjustlimits\max_{0 \leq i \leq m-1} \max_{1 \leq j \leq \floor{n/m}+1} \left| f \left( \frac{\floor{n \frac{i}{m}}}{n}+\frac{j}{n} \wedge 1 \right)- f \left( \frac{\floor{n \frac{i}{m}}}{n} \right) \right|. \label{add-neq43}
	\end{align}
	For the last line we used that \begin{equation*}
		f \left( \frac{\floor{n \frac{i}{m}}}{n} \right) = f \left( \frac{i}{m} \right)= f \left( \frac{\floor{m t}}{m} \right) \fa t \in \bigg[ \frac{i}{m},\frac{i+1}{m} \bigg)
	\end{equation*}
	as $f \in T_n(K^n)$. Combining \eqref{add-neq41} and \eqref{add-neq43}, we get \begin{align*}
		I_2
		&\leq \mbb{P} \left\{ \adjustlimits\max_{0 \leq i \leq m-1} \max_{1 \leq j \leq \floor{n/m}+1} \left| Z_n \left( \frac{\floor{n \frac{i}{m}}}{n}+\frac{j}{n} \wedge 1 \right)- Z_n \left( \frac{\floor{n \frac{i}{m}}}{n} \right) \right| >\eps  S(n)\right\} \\
		&\leq 3\sum_{i=0}^{m-1} \max_{1 \leq j \leq \floor{n/m}+1} \mbb{P} \left\{ \left| X \left(\floor{n \frac{i}{m}}+j \wedge n \right)-X \left(\floor{n \frac{i}{m}}\right) \right| > \frac{\eps S(n)}{3} \right\} .
	\end{align*}
	By Lemma~\ref{add-n07}, \begin{align*}
		I_2
		& \leq 6 \exp \left(-\frac{S(n) \lambda \eps}{3}\right) \sum_{i=0}^{m-1} \exp \left(\frac{\lambda^2}{2}  \left[\var X\left(\floor{\frac{n i}{m}} + \floor{\frac{n}{m}}+1 \right)- \var X \left( \floor{\frac{n i}{m}} \right) \right]+ \lambda^3 E_{0,n}(\lambda) \right)
	\end{align*}
	for any $0 \leq \lambda \leq \lambda_0$. Writing \begin{equation*}
		\var X_s-\var X_r = s^\gamma \left( \left[\frac{\var X_s}{s^\gamma}-\sigma^2 \right]- \left[\frac{\var X_r}{r^\gamma}-\sigma^2 \right] \right) + \frac{\var X_r}{r^{\gamma}} (s^{\gamma}-r^{\gamma})
	\end{equation*}
	it is not difficult to see that \begin{align*}
		\frac{1}{n^\gamma} \left[\var X\left(\floor{\frac{n i}{m}} + \floor{\frac{n }{m}}+1 \right)- \var X \left( \floor{\frac{n i}{m}} \right) \right]
		&\leq c \sup_{k \geq \floor{\frac{n}{m}}} \left| \frac{\var X_k}{k^\gamma}-\sigma^2 \right| + c \left( \frac{1}{m}+ \frac{1}{n} \right)
		=: \delta(n)
	\end{align*}
	for some constant $c=c(\gamma)$. Note that the first term on the right-hand side converges to $0$ as $n \to \infty$. For $\lambda := r S(n)/n^{\gamma}$, $r \geq 1$, we find \begin{align}
		I_2
		&\leq 6m \exp \left[ \frac{S(n)^2}{n^\gamma} \left(- \frac{r \eps}{3} + \frac{r^2}{2} \delta(n) +  r^3 \frac{S(n)}{n^{2\gamma}} E_{0,n}\left( r\frac{S(n)}{n^\gamma} \right)\right) \right]. \label{add-neq45}
	\end{align}
	Consequently, by \eqref{add-neq35}, \eqref{add-neq37} and \eqref{add-neq45}, \begin{align*}
		\limsup_{n \to \infty} \frac{n^{\gamma}}{S(n)^2} \log \mbb{P} \left( d \left( \frac{Z_n}{S(n)}, T_m(K^m) \right) > \eps \right)
		&\leq \max \left\{\sigma^2-\frac{r}{3}, - \frac{r \eps}{3} + \frac{c  r^2}{2m} \right\} \\
		&\xrightarrow[]{r,m \to \infty} -\infty.
	\end{align*}
	By \cite[Lemma 3.3]{feng}, $(Z_n/S(n))_{n \in \mbb{N}}$ is exponentially tight in $(D[0,1],\|\cdot\|_{\infty})$.
\end{proof}

Now we are ready to prove that $(Z_n/S(n))_{n \in \mbb{N}}$ satisfies a moderate deviation principle.

\begin{thm} \label{add-n11}
	$(Z_n/S(n))_{n \in \mbb{N}}$ satisfies a moderate deviation principle in $(D[0,1],\|\cdot\|_{\infty})$ with speed $S(n)^2/n^{\gamma}$ and good rate function $J$ defined in \eqref{add-neq31}.
\end{thm}

\begin{proof}
	For $\alpha \in \bv \cap D[0,1]$ we set \begin{equation*}
		\Lambda_n(\alpha):=\frac{n^\gamma}{S(n)^2} \log \mbb{E}\exp \left( \frac{S(n)}{n^\gamma} \int_0^1 Z_n(s) \, d\alpha(s) \right).
	\end{equation*}
	By definition, \begin{equation*}
		Z_n(s)
		= \sum_{j=0}^{n-1} X_j \I_{[j/n,(j+1)/n)}(s) + X_n \I_{\{1\}}(s)
		= \sum_{j=1}^n (X_j-X_{j-1}) \I_{[j/n,1]}(s).
	\end{equation*}
	It follows from the independence of the increments that \begin{align*}
		\Lambda_n(\alpha)
		&= \frac{n^\gamma}{S(n)^2} \log \mbb{E} \exp \left( \frac{S(n)}{n^\gamma} \sum_{j=1}^n (\alpha(1)-\alpha(j/n)) \, (X_j-X_{j-1}) \right) \\
		&= \frac{n^\gamma}{S(n)^2}  \sum_{j=1}^n \log \mbb{E} \exp \left( \frac{S(n)}{n^{\gamma}} (\alpha(1)-\alpha(j/n)) \, (X_j-X_{j-1}) \right) \\
		&= \frac{1}{2n^{\gamma}} \sum_{j=1}^n (C_j-C_{j-1}) (\alpha(1)-\alpha(j/n))^2 \\
		& + \frac{n^\gamma}{S(n)^2} \sum_{j=1}^n \int_{j-1}^j \!\! \int \bigg[ \exp \left( \frac{S(n)}{n^\gamma} (\alpha(1)-\alpha(j/n))z \right)-1-\frac{S(n)}{n^\gamma} (\alpha(1)-\alpha(j/n)) z \bigg] \, \nu(dz,dr).
	\end{align*}
	Applying Taylor's formula and using that $\var X_j = C_j + \int_0^j \!\! \int z^2 \, \nu(dz,dr)$, we get \begin{align*}
		\Lambda_n(\alpha)
		&= \frac{1}{2n^\gamma} \sum_{j=1}^n (\var X_j-\var X_{j-1}) (\alpha(1)-\alpha(j/n))^2 \\
		& \quad+ \frac{1}{6} \frac{S(n)}{n^{2\gamma}} \sum_{j=1}^n (\alpha(1)-\alpha(j/n))^3 \int_{j-1}^j \!\! \int \exp \left(\frac{S(n)}{n^\gamma}  (\alpha(1)-\alpha(j/n))\xi_j\right) z^3 \, \nu(dz,dr) \\
		&=: I_1(n)+I_2(n)
	\end{align*}
	for some intermediate value $\xi_j$ between $0$ and $z$. It follows from Abel's summation formula and \eqref{add-neq11} that \begin{align*}
		I_1(n)
		&= \frac{1}{2n^\gamma} \sum_{j=1}^n \var X_j \left( (\alpha(1)-\alpha(j/n))^2  - (\alpha(1)-\alpha(j+1/n))^2 \right) \\
		&= -\frac{1}{2} \sum_{j=1}^n \frac{\var X_j}{j^\gamma} \left(\frac{j}{n}\right)^{\gamma} \left( (\alpha(1)-\alpha((j+1)/n))^2  - (\alpha(1)-\alpha(j/n))^2 \right)  \\
		&\xrightarrow[]{n \to \infty} - \frac{\sigma^2}{2} \int_0^1 s^\gamma d\left( (\alpha(1)-\alpha(s))^2 \right).
	\end{align*}
	Note that this integral is well-defined as $\alpha \in \bv$ (hence $\alpha^2 \in \bv$). Applying integration by parts, we obtain \begin{equation*}
		\lim_{n \to \infty} I_1(n) = \frac{\gamma \sigma^2}{2} \int_0^1 s^{\gamma-1} (\alpha(1)-\alpha(s))^2 \, ds.
	\end{equation*}
	On the other hand, it is not difficult to see that $I_2(n) \xrightarrow[]{n \to \infty} 0$. Consequently,\begin{align*}
		\Lambda(\alpha) := \lim_{n \to \infty} \Lambda_n(\alpha) = \frac{\gamma \sigma^2}{2} \int_0^1 s^{\gamma-1} (\alpha(1)-\alpha(s))^2 \, ds, \qquad \alpha \in \bv \cap D[0,1].
	\end{align*}
	Obviously, $\Lambda$ is G\^ateaux differentiable (with G\^ateaux derivative in $C[0,1]$). Therefore, the claim follows from the G\"artner--Ellis theorem, see e.\,g.\ \cite[Theorem 2.1,Theorem 2.4]{acosta-lp}.
\end{proof}

In order to carry over the moderate deviation principle from $(Z_n/S(n))_{n \in \mbb{N}}$ to $(X(t \lcdot)/S(t))_{t>0}$, we need the following auxiliary result.

\begin{lem} \label{add-n13}
	$(Z_{\floor{t}}/S(\floor{t}))_{t>0}$ and $(X(t\lcdot)/S(t))_{t>0}$ are exponentially equivalent, i.\,e.\ \begin{equation*}
		\limsup_{t \to \infty} \frac{t^{\gamma}}{S(t)^2} \log \mbb{P} \left( \left\| \frac{Z_{\floor{t}}(\lcdot)}{S(\floor{t})} - \frac{X(t\lcdot)}{S(t)} \right\|_{\infty} > \eps \right)=-\infty \fa \eps>0.
	\end{equation*}
\end{lem}

\begin{proof}
	Let $\eps>0$ and $\alpha>0$. Obviously, \begin{equation}
		\left\| \frac{Z_{\floor{t}}}{S(\floor{t})} - \frac{X(t\lcdot)}{S(t)} \right\|_{\infty}
		\leq \left|\frac{S(t)}{S(\floor{t})}- 1\right| \cdot \frac{\|Z_{\floor{t}}\|_{\infty}}{S(t)} + \left\| \frac{Z_{\floor{t}}}{S(t)}-\frac{X(t\lcdot)}{S(t)} \right\|_{\infty} =: A_t+B_t. \label{add-neq51}
	\end{equation}
	We estimate $\mbb{P}(A_t>\eps)$ and $\mbb{P}(B_t>\eps)$ separately. As in the proof of Lemma~\ref{add-n09} we find \begin{align*}
		\mbb{P}(A_t>\eps)
		\leq \mbb{P} \left( \max_{0 \leq k \leq \floor{t}} \frac{|X_k|}{S(t)} > \frac{\eps}{\delta(t)} \right)
		\leq 6 \exp \left[- \frac{S(t)^2}{t^{\gamma}} \left( \frac{\eps}{3\delta(t)} - \sigma^2 \right) \right]
	\end{align*}
	for $t$ sufficiently large where $\delta(t) := \left| \frac{S(t)}{S(\floor{t})} -1 \right| \to 0$ as $t \to \infty$, cf.\ \eqref{add-neq15}. In order to estimate $B_t$ we note that \begin{align*}
			\sup_{s \in [0,1]} |Z_{\floor{t}}(s)-X(t s)|
			&\leq \sup_{\substack{u,v \leq t \\ |u-v| \leq 2}} |X_{u}-X_v|
			\leq 3 \adjustlimits\max_{j \leq \floor{\alpha^{-1}}+1} \sup_{\substack{0 \leq u \leq \alpha t \\  j\alpha t+ u \leq t}} |X_{j \alpha t+u}-X_{j \alpha t}|
	\end{align*}
	for $t \geq t_0(\alpha)$ sufficiently large. Applying again Etemadi's inequality yields \begin{align*}
		\mbb{P}(B_t>\eps)
		&\leq 3 \sum_{j=0}^{\floor{\alpha^{-1}}+1} \sup_{\substack{0 \leq u \leq \alpha t \\ j\alpha t +u \leq t}}  \mbb{P} \left( |X_{j \alpha t+u}-X_{j\alpha t}|> \frac{S(t) \eps}{9} \right).
	\end{align*}
	By Markov's inequality and Lemma~\ref{add-n07}, \begin{align}
		\mbb{P}(B_t>\eps)
		& \leq 6\exp \left(-\frac{S(t) \lambda \eps}{9} \right)  \sum_{j=0}^{\floor{\alpha^{-1}}+1} \exp \left( \frac{\lambda^2}{2} (\var X_{(j+1)\alpha t}-\var X_{j \alpha t}) + \lambda^3 E_{0,t}(\lambda) \right). \label{add-neq53}
	\end{align}
	A similar calculation as in the second part of the proof of Lemma~\ref{add-n09} shows \begin{align*}
		\frac{1}{t^{\gamma}} (\var X_{(j+1)\alpha t}-\var X_{j\alpha t}) 
		&\leq 2\alpha c+ c \sup_{k \geq 1} \left| \frac{\var X_{k\alpha t}}{(k \alpha t)^\gamma}-\sigma^2 \right|=:\delta(\alpha,t)
	\end{align*}
	for some constant $c=c(\gamma)$. In particular, $\delta(\alpha,t) \xrightarrow[]{t \to \infty} 2 \alpha c$. Setting $\lambda := r S(t)/t^\gamma$, $r \geq 1$, we get \begin{align*}
		\limsup_{t \to \infty} \frac{t^\gamma}{S(t)^2} \log \mbb{P}(B_t>\eps)
		&\leq - \frac{r \eps}{9}+ \lim_{t \to \infty} \left( \frac{r^2}{2} \delta(\alpha,t) + r^3 \frac{S(t)}{t^{2\gamma}} E_{0,t} \left( r \frac{S(t)}{t^\gamma} \right) \right)
		= - \frac{r \eps}{9} +c \alpha r^2.
	\end{align*}
	Finally, we conclude \begin{align*}
		\limsup_{t \to \infty} \frac{t^\gamma}{S(t)^2} \log \mbb{P} \left(\left\| \frac{Z_{\floor{t}}}{S(\floor{t})} - \frac{X(t\lcdot)}{S(t)} \right\|_{\infty} > 2\eps \right)
		&\leq  - \frac{r \eps}{9}+c \alpha r^2  \xrightarrow[]{\alpha \to 0,r \to \infty} -\infty. \qedhere
	\end{align*}
\end{proof}

Combining Lemma~\ref{add-n11} and Lemma~\ref{add-n13}, we find that $(X(t \lcdot)/S(t))_{t>0}$ satisfies a moderate deviation principle with good rate function $J$ and speed $S(t)^2/t^{\gamma}$, cf.\ \cite[Theorem 4.2.13]{dz}. It remains to identify the good rate function.

\begin{thm} \label{add-n15}
	The good rate function $J$, \begin{equation*}
		J(f) = \sup_{\alpha \in \bv \cap D[0,1]} \left( \int f \, d\alpha - \frac{\gamma \sigma^2}{2} \int_0^1 s^{\gamma-1} (\alpha(1)-\alpha(s))^2 \, ds \right),
	\end{equation*}
	equals \begin{equation*}
		I(f) = \begin{cases} \displaystyle\frac{1}{2\gamma \sigma^2} \int_0^1 \frac{f'(s)^2}{s^{\gamma-1}} \, ds, & f \in \ac, f(0)=0, \\ \displaystyle\infty, & \text{otherwise}. \end{cases}
	\end{equation*}
	In particular, $\dom I = \dom J \subseteq \ac$.
\end{thm}

\begin{proof}
	We only consider $\gamma \geq 1$; the case $0<\gamma<1$ is proved similarly. First, we show that $J(f)< \infty$ implies $f \in \ac$ and $f(0)=0$. If so, then \begin{equation}
		f(t) = \int_0^t f'(s) \, ds, \qquad t \in [0,1]. \label{add-neq55}
	\end{equation}
	To this end, let $0<s_1<t_1 < \ldots < s_n < t_n \leq 1$ and define \begin{equation*}
		\alpha(t) := \sum_{j=1}^n c_j \I_{[s_j,t_j)}(t), \qquad t \in [0,1]
	\end{equation*}
	for $c=(c_1,\ldots,c_n) \in \mbb{R}^n$. Obviously, $\alpha \in \bv \cap D[0,1]$ and \begin{equation}
		\int_0^1 f \, d\alpha = \sum_{j=1}^n c_j (f(s_j)-f(t_j)). \label{add-neq57}
	\end{equation}
	By definition of $J$ we have \begin{align*}
		\sum_{j=1}^n c_j (f(s_j)-f(t_j))
		&= \int_0^1 f \, d\alpha
		\leq J(f) + \frac{\gamma \sigma^2}{2} \int_0^1 s^{\gamma-1} \alpha(s)^2 \, ds
		\leq J(f)+ \frac{\gamma \sigma^2 }{2} \sum_{j=1}^n c_j^2 |t_j-s_j|.
	\end{align*}
	If we choose $c_j := r \sgn(f(s_j)-f(t_j))$, $r \geq 1$, we get \begin{equation*}
		\sum_{j=1}^n |f(s_j)-f(t_j)| \leq \frac{J(f)}{r} + \frac{r \gamma \sigma^2 }{2} \sum_{j=1}^n |t_j-s_j|.
	\end{equation*}
	This proves that $f$ is absolutely continuous. A similar calculation shows \begin{equation*}
		|f(t)| \leq \frac{J(f)}{r}+t^{\gamma} \frac{r \sigma^2}{2}.
	\end{equation*}
	Letting $t \to 0$ and $r \to \infty$ we obtain $f(0)=0$. This proves \eqref{add-neq55}.
	Now let $f$ be given by \eqref{add-neq55}. If we set $\varphi(x) := x^2/2$, then we see from the integration by parts formula that \begin{align*}
		\int_0^1 f \, d\alpha - \frac{\gamma \sigma^2}{2} \int_0^1 s^{\gamma-1}& (\alpha(1)-\alpha(s))^2 \, ds \\
		&= \int_0^1  \left[ f'(s) (\alpha(1)-\alpha(s)) - \frac{\gamma \sigma^2}{2} s^{\gamma-1} (\alpha(1)-\alpha(s))^2 \right] \, ds \\
		&= \int_0^1 \left[f'(s) (\alpha(1)-\alpha(s)) - \varphi \left( \sqrt{\gamma \sigma^2 s^{\gamma-1}}  (\alpha(1)-\alpha(s))\right) \right] \, ds \\
		&\leq \int_0^1 \varphi^{\ast} \left( \frac{f'(s)}{\sqrt{\gamma \sigma^2  s^{\gamma-1}}} \right) \, ds \\
		&=\frac{1}{2 \gamma \sigma^2 } \int_0^1 \frac{f'(s)^2}{s^{\gamma-1}} \, ds = I(f)
	\end{align*}
	where \begin{equation*}
		\varphi^{\ast}(\beta) := \sup_{x \in \mbb{R}}\left(\beta x - \varphi(x)\right) = \frac{1}{2} \beta^2, \qquad \beta \in \mbb{R},
	\end{equation*}
	denotes the Legendre transform of $\varphi$. Hence, $J(f) \leq I(f)$. On the other hand, for $f \in \ac$,\begin{equation*}
		\sum_{j=0}^{n-1} \frac{f \left( \frac{j+1}{n} \right)- f \left( \frac{j}{n} \right)}{\frac{1}{n}} \left( \frac{n}{j+1} \right)^{(\gamma-1)/2} \I_{\big[\frac{j}{n},\frac{j+1}{n} \big)}(s) \xrightarrow[]{n \to \infty} \frac{f'(s)}{s^{(\gamma-1)/2}} \qquad \text{for almost all $s \in [0,1]$}.
	\end{equation*}
	Therefore, Fatou's lemma and \eqref{add-neq57} imply \begin{align*}
		I(f)
		&\leq \frac{1}{2\sigma^2 \gamma} \liminf_{n \to \infty} \left( n \sum_{j=0}^{n-1} \left[f \left( \frac{j+1}{n} \right)- f \left( \frac{j}{n} \right) \right]^2 \left( \frac{n}{j+1} \right)^{\gamma-1} \right) \\
		&= \liminf_{n \to \infty} \bigg( \int_0^1 f \, d\alpha^{[n]} - \frac{\sigma^2 \gamma}{2} \sum_{j=0}^{n-1} \int_{j/n}^{(j+1)/n} (\alpha_j^{[n]})^2 \left( \frac{j+1}{n} \right)^{\gamma-1} \,ds \bigg)
	\end{align*}
	for \begin{align*}
		\alpha_j^{[n]} &:= - \frac{n}{\sigma^2 \gamma}  \left( \frac{n}{j+1} \right)^{\gamma-1} \left[f \left( \frac{j+1}{n} \right)- f \left( \frac{j}{n} \right) \right] \\
		\alpha^{[n]}(t) &:= \sum_{j=0}^{n-1} \alpha_j^{[n]} \I_{[j/n,(j+1)/n)}(t).
	\end{align*}
	Using that $\left(\frac{j+1}{n} \right)^{\gamma-1} \geq s^{\gamma-1}$ for any $s \in [j/n,(j+1)/n]$, we get  \begin{equation*}
		I(f)
		\leq \liminf_{n \to \infty} \left( \int_0^1 f \, d\alpha^{[n]} - \frac{\sigma^2 \gamma}{2} \int_0^1  \alpha^{[n]}(s)^2 s^{\gamma-1} \, ds \right)
		\leq J(f). \qedhere
	\end{equation*}
\end{proof}

Next, we prove Strassen's law, Theorem~\ref{add-n03}. The proof is inspired by \cite[Chapter 13]{bm2} where the result is shown for Brownian motion. Without loss of generality, we assume throughout the proof of Theorem~\ref{add-n03} that $\sigma^2=1$. We need the following lemma.

\begin{lem} \label{add-n17}
Let $I$ be the good rate function defined in \eqref{add-neq23}. For $c>0$ set \begin{align*}
	A &:= \left\{f \in D[0,1]; \adjustlimits\sup_{q^{-1} \leq t \leq 1} \sup_{0 \leq s \leq 1} |f(st)-f(s)| \geq c \right\}, & B := \left\{f \in D[0,1]; \sup_{s\in [0,1]} |f(s)| \geq c \right\}.
\end{align*}
Then $A$ and $B$ are closed  in $(D[0,1],\|\cdot\|_{\infty})$ and \begin{align*}
	\inf_{f \in A} I(f) \geq \frac{c^2}{2} \frac{q^{\gamma}}{q^{\gamma}-1}, \qquad \quad \inf_{f \in B} I(f) \geq \frac{c^2}{2}.
\end{align*}
\end{lem}

\begin{proof}
	By the Cauchy--Schwarz inequality, \begin{align*}
		c^2
		\leq \adjustlimits\sup_{q^{-1} \leq t \leq 1} \sup_{0 \leq s \leq 1} |f(st)-f(s)|^2
		&= \adjustlimits\sup_{q^{-1} \leq t \leq 1} \sup_{0 \leq s \leq 1} \left| \int_{st}^{s} f'(r)\, dr \right|^2 \\
		&\leq \adjustlimits\sup_{q^{-1} \leq t \leq 1} \sup_{0 \leq s \leq 1} \left[ \left( \int_0^1 \frac{f'(r)^2}{r^{\gamma-1}} \, dr \right) \left(\int_{st}^s r^{\gamma-1} \,dr \right) \right] \\
		&= 2 (1-q^{-\gamma}) I(f)
	\end{align*}
	for any $f \in A \cap \dom I \subseteq \ac$. Similarly, we find from the Cauchy--Schwarz inequality\begin{equation*}
		c^2
		\leq \sup_{0 \leq s \leq 1} \left| \int_0^s f'(r) \, dr \right|^2
		\leq 2\gamma I(f) \int_0^1 r^{\gamma-1} \, dr = 2 I(f) \fa f \in B. \qedhere
	\end{equation*}
\end{proof}

\begin{lem} \label{add-n19}%
	The set of limit points $\mc{L}(\omega)$ satisfies $\mc{L}(\omega) \subseteq \Phi(\frac{1}{2})$ for almost all $\omega \in \Omega$.
\end{lem}

\begin{proof} %
	Since $\Phi(\frac{1}{2}) = \bigcap_{r>0} \Phi(\frac{1}{2}+r)$, it suffices to show $\mc{L}(\omega) \subseteq \Phi(\frac{1}{2}+r)$ for any $r>0$. Set $Z_t := X(t \lcdot)/S(t)$, and fix $q>1$, $\delta>0$. Applying the large deviation upper bound (cf.\ \cite[Theorem 3.3.3]{fw}) gives \begin{align*}
		\mbb{P} \left( d\big(Z_{q^n},\Phi(\tfrac{1}{2}+r)\big)>\delta \right)
		&\leq \exp \left(-2 \left( \frac{1}{2}+\frac{r}{2} \right) \log \log q^n \right)
	\end{align*}
	for $n$ sufficiently large. Thus, by the Borel--Cantelli lemma, \begin{equation*}
		d\big(Z_{q^n}(\lcdot,\omega),\Phi(\tfrac{1}{2}+r)\big) \leq \delta \end{equation*}
	for $n \geq n_0(q,\omega)$. It remains to fill the gaps in the sequence $(q^n)_{n \in \mbb{N}}$. Note that \begin{align*}
		\sup_{q^{n-1}\leq t\leq q^n} \|Z_t - Z_{q^n}\|_\infty
		&= \adjustlimits\sup_{q^{n-1}\leq t\leq q^n}\sup_{0\leq s\leq 1}\left|\frac{X(st)}{S(t)} - \frac{X(sq^n)}{S(q^n)}\right|\\
		&\leq \adjustlimits\sup_{q^{n-1}\leq t\leq q^n}\sup_{0\leq s\leq 1} \frac{|X(st)-X(sq^n)|}{S(q^n)}
		+ \adjustlimits\sup_{q^{n-1}\leq t\leq q^n}\sup_{0\leq s\leq 1}  \frac{|X(st)|}{S(q^n)}\left|\frac{S(q^n)}{S(t)}-1\right| \\
		&=: A_n+B_n.
	\end{align*}
	As \begin{equation*}
		\mathbb{P}\left(B_n \geq \frac{\delta}{2} \right)
		= \mathbb{P} \left( \sup_{0 \leq t \leq 1} \left| \frac{X(t q^n)}{S(q^n)} \right| \cdot \left| \frac{S(q^n)}{S(q^{n-1})}-1 \right| \geq \frac{\delta}{2} \right),
	\end{equation*}
	it follows easily from $S(q^n)/S(q^{n-1}) \xrightarrow[]{n \to \infty} q^{\gamma/2}$ and the large deviation upper bound that $\sum_{n \in \mbb{N}} \mathbb{P}(B_n \geq \delta/2)<\infty$ if we choose $q>1$ close to $1$. Hence, by the Borel--Cantelli lemma, $B_n \leq \frac{\delta}{2}$ for $n \geq n_1(q,\omega)$ sufficiently large. In order to estimate $A_n$ we note that\begin{align*}
		\mbb{P}\left( A_n \geq \frac{\delta}{2} \right)
		&= \mbb{P}\left( \adjustlimits\sup_{q^{-1} \leq t \leq 1} \sup_{0 \leq s \leq 1} \frac{|X(s q^n t)-X( sq^n)|}{S(q^n)} \geq \frac{\delta}{2} \right)
		\leq \mbb{P} \left( \frac{X(q^n \lcdot)}{S(q^n)} \in A \right)
	\end{align*}
	for \begin{equation*}
		A := \left\{f \in D[0,1]; \adjustlimits\sup_{q^{-1} \leq t \leq 1} \sup_{0 \leq s \leq 1} |f(st)-f(s)| \geq \frac{\delta}{2} \right\}.
	\end{equation*}
	Therefore, by the large deviation upper bound and Lemma~\ref{add-n17}, \begin{equation*}
		\mbb{P}\left( A_n\geq \frac{\delta}{2} \right) \leq \exp \left(- \frac{\delta^2}{8} \frac{q^\gamma}{q^\gamma-1}  \log \log q^n \right)
	\end{equation*}
	for  $n \geq n_2(q)$ sufficiently large.
	If $q>1$ is close to $1$, this implies $\sum_{n \in \mbb{N}} \mbb{P}(A_n\geq\delta/2)<\infty$. By the Borel--Cantelli lemma, we conclude \begin{equation*}
		\sup_{q^{n-1} \leq t \leq q^n} \|Z_t-Z_{q^n}\|_{\infty} \leq \delta
	\end{equation*}
	for $n \geq n_3(q,\omega)$ sufficiently large. Finally,\begin{align*}
		d\big(Z_s(\lcdot,\omega),\Phi(\tfrac 12+r)\big)
		\leq\|Z_s(\lcdot,\omega)-Z_{q^n}(\lcdot,\omega)\|_\infty + d\big(Z_{q^n}(\lcdot,\omega),\Phi(\tfrac 12+r)\big)
		\leq 2\delta
	\end{align*}
	for $s$ sufficiently large. Since $\Phi(\frac{1}{2}+r)$ is closed, this proves the claim.
\end{proof}
		
	\begin{lem} \label{add-n21}%
		$\mc{L}(\omega) \supseteq \Phi(\frac{1}{2})$ for almost all $\omega \in \Omega$.
	\end{lem}
	
	\begin{proof} %
		Since the sublevel sets are compact, we have $\overline{\bigcup_{r<1/2} \Phi(r)} \subseteq \Phi(\frac{1}{2})$. On the other hand, any $f \in \Phi(\frac{1}{2})$ can be approximated by $(1-\eps) f \in \Phi(\frac{(1-\eps)^2}{2})$. Therefore, $\overline{\bigcup_{r<1/2} \Phi(r)} = \Phi(\frac{1}{2})$. Consequently, it suffices to show that for any $r<\frac{1}{2}$, $\eps>0$, $f \in \Phi(r)$ there is a.\,s.\ a sequence $s_n=s_n(\omega) \to \infty$ such that
		\begin{equation*}
			\limsup_{n \to \infty} \|Z_{s_n}(\omega)-f\|_{\infty} \leq \eps.
		\end{equation*}
		Pick $q>1$. Obviously, \begin{align*}
		 	\|Z_{q^n}-f\|_{\infty}
		 	&\leq \sup_{q^{-1} \leq t \leq 1} \left| \frac{X(t q^n)-X(q^{n-1})}{S(q^n)} - f(t) \right| + \left|\frac{X(q^{n-1})}{S(q^n)}\right| + \sup_{t \leq q^{-1}} |f(t)| + \sup_{t \leq q^{-1}} \left|\frac{X(t q^n)}{S(q^n)}\right|.
		 \end{align*}
		We estimate the terms separately. Setting \begin{align*}
			A := \left\{g \in D[0,1]; \sup_{q^{-1} \leq t \leq 1} |g(t)-g(q^{-1})-f(t)| < \frac{\eps}{4} \right\}
		\end{align*}
		we have \begin{align*}
			\mbb{P}(A_n) := \mbb{P} \left( \sup_{q^{-1} \leq t \leq 1} \left| \frac{X(t q^n)-X(q^{n-1})}{S(q^n)} - f(t) \right| < \frac{\eps}{4} \right)
			&= \mbb{P} \left( \frac{X(q^n \lcdot)}{S(q^n)} \in A \right).
		\end{align*}
		If we choose $q>1$ sufficiently large such that $|f(q^{-1})|<\frac{\eps}{4}$, then $f \in A$. By assumption, $I(f) < \frac{1}{2}$ and therefore we conclude from the large deviation lower bound that $\sum_{n \in \mbb{N}} \mbb{P}(A_n)=\infty$. Taking a subsequence, if necessary, we obtain by applying the Borel--Cantelli lemma \begin{equation*}
			\adjustlimits\limsup_{n \to \infty} \sup_{q^{-1} \leq t \leq 1} \left| \frac{X(t q^n)-X(q^{n-1})}{S(q^n)} - f(t) \right| \leq \frac{\eps}{4}.
		\end{equation*}
		By H\"older's inequality, \begin{equation*}
			\sup_{t \leq q^{-1}} |f(t)|^2
			= \sup_{t \leq q^{-1}}\left|\int_0^{t} f'(s) \, ds\right|^2
			\leq \left( \int_0^1 \frac{f'(s)^2}{s^{\gamma-1}} \, ds \right) \left( \int_0^{q^{-1}} s^{\gamma-1} \, ds \right)
			\leq \frac{1}{q^{\gamma}}.
		\end{equation*}
		Moreover, \begin{align*}
			\mbb{P} \left( \left| \frac{X(q^{n-1})}{S(q^n)} \right| \geq \frac{\eps}{4} \right)+\mbb{P} \left( \sup_{0 \leq t \leq q^{-1}} \left| \frac{X(t q^n)}{S(q^n)} \right| \geq \frac{\eps}{4} \right)
			&\leq 2 \mathbb{P} \left( \sup_{0 \leq t \leq 1} \left|\frac{X(q^{n-1}t)}{S(q^{n-1})} \right| \frac{S(q^{n-1})}{S(q^{n})} \geq \frac{\eps}{4} \right).
			\end{align*}
		By Lemma~\ref{add-n17}, it is not difficult to see that we may apply again the Borel--Cantelli lemma if we choose $q>1$ sufficiently large. Hence,
		\begin{equation*}
			\limsup_{n \to \infty} \left(\left|\frac{X(q^{n-1})}{S(q^n)}\right| + \sup_{t \leq q^{-1}} |f(t)| + \sup_{t \leq q^{-1}} \left|\frac{X(t q^n)}{S(q^n)}\right| \right) \leq \frac{3}{4} \eps.
		\end{equation*}
		This finishes the proof.
	\end{proof}
	
	Combining  Lemma~\ref{add-n19} and Lemma~\ref{add-n21} yields Theorem~\ref{add-n03}. Finally, it remains to prove Corollary~\ref{add-n05}.
	
	\begin{proof}[Proof of Corollary~\ref{add-n05}]
		It follows from the assumptions that there exists an additive process $(Y_t)_{t \geq 0}$ satisfying the assumptions of Theorem~\ref{add-n01} such that $((X(t \lcdot)-Y(t \lcdot))/S(t))_{t > 0}$ converges uniformly in $D[0,1]$ to $0$ as $t \to \infty$. Therefore, the claim follows from Theorem~\ref{add-n03}. For more details, we refer the reader to \cite[Proposition 3.1,3.2]{wang}.
	\end{proof}

\section{Concluding remarks} \label{s-rem} \begin{enumerate}
	\item Theorem~\ref{add-n01} still holds if $(X_t)_{t \geq 0}$ has also fixed jump discontinuities. Taking a close look at the proof reveals that we simply have to modify the estimate of the moment generating function in Lemma~\ref{add-n07} appropriately. The corresponding estimate follows from the explicit formula for the exponential moments, cf.\ \cite{fuji}, and well-known elementary inequalities.
	\item Let $(X_t)_{t \geq 0}$ be an additive process such that $\mbb{E}X_t=0$ and \begin{equation*}
		\lim_{t \to \infty} \frac{\var X_t}{h(t)} =\sigma^2>0
	\end{equation*}
	for a regulary varying function $h:(0,\infty) \to (0,\infty)$ of index $\gamma >0$ (see e.\,g.\ \cite{bing}  for the definition). If we replace $t^{\gamma}$ in Theorem~\ref{add-n01} and Theorem~\ref{add-n03} by $h(t)$, then both theorems remain valid.
\end{enumerate}

\begin{ack}
	Financial support through the \emph{Deutsche Forschungsgemeinschaft}, DFG, grant SCHI 419/5--2 is gratefully acknowledged. 
\end{ack}

\end{document}